\documentclass{amsart}

\usepackage{amsmath, amssymb, amsthm}

\newtheorem{theorem}{Theorem}[section]
\newtheorem{proposition}[theorem]{Proposition}
\newtheorem{lemma}[theorem]{Lemma}
\newtheorem{corollary}[theorem]{Corollary}

\numberwithin{equation}{section}
\allowdisplaybreaks[0]

\newcommand \N {\mathbb N}
\newcommand \F {\mathcal F}
\newcommand \Z {\mathbb Z}
\newcommand \hocolim {\operatorname{hocolim}}
\newcommand \id {\mathrm{id}}
\newcommand \image {\operatorname{image}}
\newcommand \oc [1] {\lVert #1 \rVert}
\newcommand \pt {\mathbf{pt}}
\newcommand \rem {\mathrm{rem}}
\newcommand \suspjoin {\circledast}

\allowdisplaybreaks[4]

\title
[Frobenius complexes]
{Frobenius complexes and
the homotopy colimit of a diagram of posets over a poset}
\author
[TOUNAI S.]
{TOUNAI Shouta}
\address{Graduate School of Mathematical Sciences,
  The University of Tokyo,
3-8-1 Komaba, Meguro, Tokyo, 153-8914 Japan}
\email{tounai@ms.u-tokyo.ac.jp}

\begin{document}

\begin{abstract}
  An affine monoid is an additive monoid which is
  cancellative, pointed and finitely generated.
  An affine monoid $\Lambda$ has the partial order
  defined by $\lambda \le \lambda + \mu$.
  The Frobenius complex is the order complex
  of an open interval of $\Lambda$
  with respect to this partial order.
  The reduced homology of the Frobenius complex
  is related to the torsion group of the monoid algebra $K[\Lambda]$.
  In this paper, we pay attention to homotopy types of
  Frobenius complexes, and we express the homotopy types
  of the Frobenius complexes of $\Lambda$
  in terms of those of $\Lambda_1$ and $\Lambda_2$
  when $\Lambda$ is an affine monoid obtained by
  gluing two affine monoids $\Lambda_1$ and $\Lambda_2$
  with one relation.
  We also state an application to the Poincar{\'e} series
  of the torsion group of the monoid algebra.
\end{abstract}

\maketitle

\section{Introduction}

We consider an additive monoid
which is cancellative, pointed and finitely generated,
which we call an affine monoid.
An affine monoid $\Lambda$ has the partial order $\le$ defined by
$\lambda \le \lambda + \mu$ for elements $\lambda$ and $\mu$ of $\Lambda$.
For a non-zero element $\lambda$ of $\Lambda$
the Frobenius complex $\F(\lambda; \Lambda)$ is
the order complex $\oc{(0, \lambda)_\Lambda}$ of
the open interval of $\Lambda$.

Frobenius complex is introduced by Laudal and Sletsj{\o}e~\cite{LS},
and they showed the isomorphism
\begin{equation*}
  \operatorname{Tor}^{K[\Lambda]}_{i, \lambda}(K, K)
  \cong \Tilde H_{i - 2} \bigl( \F(\lambda; \Lambda); K \bigr),
\end{equation*}
where $K$ is a field.
We pay attention to the Poincar{\'e} series
\begin{align*}
  P_\Lambda(t, z)
  &= P^{K[\Lambda]}_K(t, z) \\
  &= \sum_{i \in \N} \sum_{\lambda \in \Lambda}
  \dim_K \operatorname{Tor}^{K[\Lambda]}_{i, \lambda}(K, K) \cdot
  t^i z^\lambda
\end{align*}
of this torsion group.

Clark and Ehrenborg~\cite{CE} took notice of homotopy types of
Frobenius complexes,
and determined the homotopy types of the Frobenius complexes
of $\Lambda$ by using discrete Morse theory
when $\Lambda$ is the submonoid $\langle a, b \rangle$ of $\N$
generated by two relatively prime integers $a$ and $b$
\cite[Theorem~4.1]{CE}, and
when $\Lambda$ is the submonoid
$\langle\, a + n d \,|\, n \in \N \,\rangle$ of $\N$
generated by the arithmetic sequence for
two relatively prime integers $a$ and $d$ \cite[Theorem~5.1]{CE}.

Tounai~\cite{Tou} expressed the homotopy types
of the Frobenius complexes of $\Lambda[\rho / r]$
in terms of the Frobenius complexes of $\Lambda$,
where $\Lambda[\rho / r]$ denotes the affine monoid
obtained from $\Lambda$ by adjoining the formal $r$-th part
of a reducible element $\rho$ of $\Lambda$ \cite[Theorem~3.1]{Tou}.
This result is an extension of \cite[Theorem~4.1]{CE},
and the proof is based on basic tools of homotopy theory
for posets and CW complexes.

The main theorem of this paper is an extension of
\cite[Theorem~3.1]{Tou},
and express the homotopy types of the Frobenius complexes
of $\Lambda$ in terms of the Frobenius complexes of
$\Lambda_1$ and $\Lambda_2$
when $\Lambda$ is the affine monoid obtained from
the direct sum $\Lambda_1 \oplus \Lambda_2$ of
two affine monoids
by identifying a reducible element $\rho_1$ of $\Lambda_1$
with a reducible element $\rho_2$ of $\Lambda_2$.
As a corollary,
the Poincar{\'e} series of $\Lambda$ is expressed as
\begin{equation*}
  P_{\Lambda}(t, z) =
  \frac{P_{\Lambda_1}(t, z) \cdot P_{\Lambda_2}(t, z)}{1 - t^2 z^\rho},
\end{equation*}
where $\rho$ denotes the equivalence class
of $\rho_1$ and $\rho_2$ in $\Lambda$.
The case $z = 1$ of this formula is a special case of
\cite[Proposition~3.3.5.(2)]{Avr}, and
Avramov~\cite{Avr} implies that the graded version
of this proposition will be also proven in an algebraic way.

To prove the main theorem
we use the notion of homotopy colimit,
which is established by Bousfield and Kan~\cite{BK}.
We refer to \cite{BWW}, \cite{Tho}, \cite{WZZ} and \cite{ZZ}
about homotopy colimits.

\section{Preliminaries}

\subsection{Posets}

A \emph{partially ordered set} (\emph{poset} for short)
is a set $P$ together with a partial order $\le$.
The \emph{order complex} $\oc P$ of a poset $P$ is
the simplicial complex whose vertices are elements of $P$
and whose simplices are non-empty finite chains of $P$.
For posets $P$ and $Q$
a map $f \colon P \to Q$ is a \emph{poset map}
if $p \le p'$ implies $f(p) \le f(p')$ for any $p, p' \in P$.
The simplicial map induced by a poset map $f$
is denoted by $\oc f$.

\begin{proposition}
  [Quillen~{\cite[1.3]{Qui}}]
  Let $f, g \colon P \to Q$ be poset maps.
  If $f \le g$ holds, that is,
  $f(p) \le g(p)$ holds for each $p \in P$,
  then the induced maps $\oc f$ and $\oc g$
  are homotopic.
\end{proposition}

\begin{lemma}
  If a poset $P$ has a maximum element,
  then the order complex $\oc P$ is contractible.
\end{lemma}
\begin{proof}
  The constant map $c$ to the maximum element satisfies
  $\id_P \le c$, and thus the identity map on $\oc P$
  is homotopic to a constant map.
\end{proof}

\begin{lemma}
  \label{lem:deform}
  Let $P$ be a poset and $S$ a subset of $P$.
  If $S^{\le p} = \{\, s \in S \mid s \le p \,\}$
  has a maximum element for each $p \in P$,
  then $\oc S$ is a deformation retract of $\oc P$.
\end{lemma}
\begin{proof}
  Let $i \colon S \to P$ be the inclusion map.
  Define $r \colon P \to S$ by
  \begin{equation*}
    r(p) = \max S^{\le p}
    \quad (p \in P).
  \end{equation*}
  Then $r$ is a poset map, and
  satisfies $r \circ i = \id_S$ and $i \circ r \le \id_P$,
  which proves the lemma.
\end{proof}

\begin{theorem}
  [Quillen~{\cite[1.6]{Qui}}]
  \label{thm:QFL}
  Let $f \colon P \to Q$ be a poset map.
  If $\oc{f^{-1}(Q_{\ge q})}$ is contractible for each $q \in Q$,
  then the induced map $\oc f \colon \oc P \to \oc Q$ is
  a homotopy equivalence.
\end{theorem}

The following is a slightly modified version of \cite[Theorem~2.5]{BWW},
and is proved in the same way.

\begin{theorem}
  [Bj{\"o}rner-Wachs-Welker~{\cite[Theorem~2.5]{BWW}}]
  \label{thm:PFT}
  Let $f \colon P \to Q$ be a poset map,
  and assume that $Q$ has a minimum element $m$.
  If the inclusion
  $\oc{f^{-1}(Q^{< q})} \hookrightarrow \oc{f^{-1}(Q^{\le q})}$
  is homotopic to the constant map to a point $c_q$
  of $\oc{f^{-1}(Q^{\le q})}$ for each non-minimum element $q$ of $Q$,
  then
  \begin{equation*}
    \oc P \simeq \bigvee_{q \in Q} \oc{Q_{> q}} * \oc{f^{-1}(Q^{\le q})},
  \end{equation*}
  where the wedge is formed by identifying
  $q \in \oc{Q_{> m}} * \oc{f^{-1}(Q_{\le m})}$ with
  $c_q \in \oc{Q_{> q}} * \oc{f^{-1}(Q_{\le q})}$
  for each non-minimum element $q$ of $Q$.
\end{theorem}

\begin{lemma}
  [Walker~{\cite[Theorem~5.1.(d)]{Wal}}]
  \label{lem:Walker}
  Let $P$ and $Q$ be posets,
  $p, p' \in P$ and $q, q' \in Q$.
  Assume that $p < p'$ and $q < q'$.
  Then the order complex
  $\oc{\bigl( (p, q), (p', q') \bigr)_{P \times Q}}$
  of the open interval of the product poset
  is homeomorphic to the suspension of the join
  $\oc{(p, p')_P} * \oc{(q, q')_Q}$
  of the order complexes of the open intervals.
\end{lemma}

\subsection{Homotopy colimits over a finite poset}

A \emph{diagram} over a poset $P$ (\emph{$P$-diagram} for short)
in a category $\mathcal C$ is a functor $D \colon P \to \mathcal C$.
For $p, q \in P$ with $p \le q$
the induced morphism $D(p) \to D(q)$ of $\mathcal C$
is denoted by $D^p_q$.

Let $D$ be a diagram of topological spaces over a finite poset $P$,
that is, $D$ is a $P$-diagram in the category of topological spaces.
The \emph{homotopy colimit} of $D$ is defined by
\begin{equation*}
  \hocolim D =
  \frac{\coprod_{p \in P} \oc{P_{\ge p}} \times D(p)}{\sim},
\end{equation*}
where $\sim$ is the equivalence relation generated by
\begin{equation*}
  \oc{P_{\ge p}} \times D(p) \ni (a, x)
  \sim (a, D^p_q(x)) \in \oc{P_{\ge q}} \times D(q)
\end{equation*}
for $p, q \in P$ with $p \le q$,
$a \in \oc{P_{\ge q}}$, and $x \in D(p)$.

Note that a natural homeomorphism $\alpha \colon D \to E$
between diagrams induces a homeomorphism
$\hocolim D \approx \hocolim E$.
Moreover, it is known that a natural homotopy equivalence induces
a homotopy equivalence $\hocolim D \simeq \hocolim E$ \cite[3.7]{WZZ}.

For a $P$-diagram $X$ of posets
by composing the order complex functor
we obtain a $P$-diagram of topological spaces,
denoted by $\oc X$.

\begin{theorem}
  [Thomason~{\cite[Theorem~1.2]{Tho}}]
  \label{thm:Thomason}
  Let $X$ be a diagram of posets over a finite poset $P$.
  Then the homotopy colimit of $\oc X$ is homotopy equivalent to
  the order complex of the poset $P \ltimes X$ defined by
  \begin{align*}
    P \ltimes X
    &= \coprod_{p \in P} X(p) \\
    &= \bigl\{\, (p, x) \,\bigm|\, p \in P, \ x \in X(p) \,\bigr\}
  \end{align*}
  and
  \begin{equation*}
    (p, x) \le (q, y)
    \iff \text{$p \le q$ and $X^p_q(x) \le y$.}
  \end{equation*}
\end{theorem}

\subsection{Frobenius complexes}

In this paper $\N$ denotes the additive monoid $\Z_{\ge 0}$
of non-negative integers.
An \emph{affine monoid} is an additive monoid $\Lambda$
which satisfies the three following condition.
\begin{enumerate}
  \item
    $\Lambda$ is cancellative, that is,
    $\lambda + \mu = \lambda' + \mu$ implies $\lambda = \lambda'$
    for any $\lambda, \lambda', \mu \in \Lambda$.
  \item
    $\Lambda$ is pointed, that is,
    $\lambda + \mu = 0$ implies $\lambda = \mu = 0$
    for any $\lambda, \mu \in \Lambda$.
  \item
    $\Lambda$ is finitely generated, that is,
    there exist finite elements $\alpha_1, \dots, \alpha_d$ of $\Lambda$
    such that for any element $\lambda$ of $\Lambda$ can be written as
    $\lambda = m_1 \alpha_1 + \dots + m_d \alpha_d$
    for some $m_1, \dots, m_d \in \N$.
\end{enumerate}
For example, a finitely generated submonoid of $\N^d$ is an affine monoid.

An affine monoid $\Lambda$ has the \emph{Frobenius order} $\le$ defined by
\begin{equation*}
  \lambda \le \nu \iff
  \text{there exists $\mu \in \Lambda$ satisfying $\lambda + \mu = \nu$.}
\end{equation*}

For a non-zero element $\lambda$ of $\Lambda$
we define the \emph{Frobenius complex} $\F(\lambda; \Lambda)$ by
\begin{equation*}
  \F(\lambda; \Lambda) = \oc{(0, \lambda)_\Lambda},
\end{equation*}
where $(0, \lambda)_\Lambda$ denotes the open interval of $\Lambda$
with respect to the Frobenius order.
Note that $(0, \lambda)_\Lambda$ is finite~\cite[Proposition~2.5]{Tou}.
We also define
\begin{equation*}
  \F(0; \Lambda) = S^{-2},
\end{equation*}
where $S^{-2}$ is just a formal symbol, not a topological space.

\begin{proposition}
  \label{prp:N}
  The Frobenius complexes of $\N$ satisfy
  \begin{align*}
    \F(0; \N) &= S^{-2} \\
    \F(1; \N) &= S^{-1} \\
    \F(n; \N) &= \pt \quad (n \ge 2).
  \end{align*}
\end{proposition}
\begin{proof}
  The case $n = 0$ follows by definition,
  and the case $n = 1$ from $(0, 1)_\N = \emptyset$.
  If $n \ge 2$, then we have
  \begin{equation*}
    \F(n; \N) = \oc{(0, n)_\N} = \oc{[1, n)_\N} \simeq \pt.
    \qedhere
  \end{equation*}
\end{proof}

\begin{lemma}
  \label{lem:suspjoin}
  Let $\Lambda_1$ and $\Lambda_2$ be affine monoids.
  Then
  \begin{equation*}
    \F(\lambda_1 + \lambda_2; \Lambda_1 \oplus \Lambda_2)
    \approx \F(\lambda_1; \Lambda_1) \suspjoin \F(\lambda_2; \Lambda_2)
  \end{equation*}
  for $\lambda_1 \in \Lambda_1$ and $\lambda_2 \in \Lambda_2$,
  where $X \suspjoin Y$ denotes
  the suspension of the join $X * Y$
  for topological spaces $X$ and $Y$,
  and let $S^{-2} \suspjoin X = X = X \suspjoin S^{-2}$.
\end{lemma}
\begin{proof}
  It follows from Lemma~\ref{lem:Walker}.
\end{proof}

\begin{lemma}
  \label{lem:remainder}
  Let $\Lambda$ be an affine monoid and
  $\rho$ a non-zero element of $\Lambda$.
  Then for any $\lambda \in \Lambda$ there uniquely exist $\ell \in \N$ and
  $\Hat \lambda \in \Lambda$ which satisfy
  $\ell \rho + \Hat \lambda = \lambda$ and $\Hat \lambda \ngeq \rho$.
\end{lemma}
\begin{proof}
  Let $\lambda \in \Lambda$.
  By \cite[Lemma~2.8]{Tou},
  there exists a maximum $\ell$ satisfying $\ell \rho \le \lambda$.
  Then we can take $\Hat \lambda \in \Lambda$ satisfying
  $\ell \rho + \Hat \lambda = \lambda$.
  The maximality of $\ell$ implies $\Hat \lambda \not\ge \rho$.
  We next show the uniqueness.
  Let $\ell \rho + \Hat \lambda = \ell' \rho + \Hat \lambda'$
  and $\Hat \lambda, \Hat \lambda' \not\ge \rho$.
  We can assume that $\ell \le \ell'$.
  Cancellativity implies $\Hat \lambda = (\ell' - \ell) \rho + \Hat \lambda'$.
  Since $\Hat \lambda \not\ge \rho$, we have $\ell = \ell'$,
  and thus $\Hat \lambda = \Hat \lambda'$.
\end{proof}

The \emph{composition poset} $C(\lambda; \Lambda)$ is the poset
of non-trivial ordered partitions of $\lambda$ in $\Lambda$,
that is,
\begin{equation*}
  C(\lambda; \Lambda) =
  \biggl\{\, [\xi^{(1)}|\dotsb|\xi^{(k)}] \,\biggm|\,
    k \ge 2, \ \xi^{(i)} \in \Lambda \setminus \{ 0 \}, \
  \sum_{i = 1}^k \xi^{(i)} = \lambda \,\biggr\}
\end{equation*}
and the order is generated by
\begin{equation*}
  [\xi^{(1)}|\dotsb|\xi^{(i)} + \xi^{(i + 1)}|\dotsb|\xi^{(k)}]
  \le [\xi^{(1)}|\dotsb|\xi^{(i)}|\xi^{(i + 1)}|\dotsb|\xi^{(k)}]
\end{equation*}
for $[\xi^{(1)}|\dotsb|\xi^{(k)}] \in C(\lambda; \Lambda)$ with $k \ge 3$
and $i = 1, \dots, k - 1$.

\begin{proposition}
  \label{prp:comp}
  For an affine monoid $\Lambda$ and
  a non-zero element $\lambda$ of $\Lambda$
  the Frobenius complex $\F(\lambda; \Lambda)$ is homeomorphic to
  the order complex of the composition poset $C(\lambda; \Lambda)$.
\end{proposition}
\begin{proof}
  Define the map $\Phi$ from $C(\lambda; \Lambda)$ to
  the face poset of $\oc{(0, \lambda)_\Lambda}$ by
  \begin{equation*}
    \Phi \bigl( [\xi^{(1)}|\dotsb|\xi^{(k)}] \bigr) =
    \bigl \{ \xi^{(1)} < \xi^{(1)} + \xi^{(2)} < \dots <
    \xi^{(1)} + \dots + \xi^{(k - 1)} \bigr \}.
  \end{equation*}
  Then the inverse of $\Phi$ is given by
  \begin{equation*}
    \Phi^{-1} \bigl( \{ \mu_0 < \dots < \mu_k \} \bigr) =
    \bigl[ \mu_0 | \mu_1 - \mu_0 | \dotsb
    | \mu_k - \mu_{k - 1} | \lambda - \mu_k \bigr].
  \end{equation*}
  We can check that both $\Phi$ and $\Phi^{-1}$ are order-preserving.
  Thus $\oc{C(\lambda; \Lambda)}$ is isomorphic to
  the barycentric subdivision of $\oc{(0, \lambda)_\Lambda}$.
\end{proof}

A monoid homomorphism $\varphi \colon \Lambda \to \Lambda'$
between two affine monoids is \emph{proper} if
$\varphi(\lambda) = 0$ implies $\lambda = 0$ for any $\lambda \in \Lambda$.
Note that if $\varphi$ is proper, then $\varphi$ is strictly order-preserving,
but not necessarily injective.
A proper homomorphism $\varphi \colon \Lambda \to \Lambda'$ induces poset maps
\begin{align*}
  \varphi &\colon (0, \lambda)_\Lambda
  \to \bigl( 0, \varphi(\lambda) \bigr)_{\Lambda'} \\
  \varphi_* &\colon C(\lambda; \Lambda)
  \to C \bigl( \varphi(\lambda); \Lambda' \bigr)
\end{align*}
for a non-zero element $\lambda$ of $\Lambda$
in obvious ways.

\begin{lemma}
  \label{lem:going down}
  Let $\varphi \colon \Lambda \to \Lambda'$ be
  a proper homomorphism between affine monoids,
  $\lambda$ a non-zero element of $\Lambda$,
  and $\lambda' = \varphi(\lambda) \in \Lambda'$.
  For $\xi \in C(\lambda; \Lambda)$ and $\eta \in C(\lambda'; \Lambda')$
  which satisfy $\varphi_*(\xi) \ge \eta$
  there uniquely exists $\xi' \in C(\lambda; \Lambda)$ satisfying
  $\xi \ge \xi'$ and $\varphi_*(\xi') = \eta$.
\end{lemma}
\begin{proof}
  We use the poset isomorphism $\Phi$ defined
  in the proof of Proposition~\ref{prp:comp}.
  Let $\{ \mu_1 < \dots < \mu_k \}$ and $\{ \mu'_1 < \dots < \mu'_{k'} \}$
  be non-empty chains of $(0, \lambda)_\Lambda$ and
  $(0, \lambda')_{\Lambda'}$, respectively.
  Assume that $\varphi(\{ \mu_1, \dots, \mu_k \})$ contains
  $\{ \mu'_1, \dots, \mu'_{k'} \}$.
  Since $\varphi$ is strictly order-preserving,
  there uniquely exists a subchain of $\{ \mu_1, \dots, \mu_k \}$
  whose image by $\varphi$ coincides with $\{ \mu'_1, \dots, \mu'_{k'} \}$.
\end{proof}

An \emph{additive equivalence relation} on an additive monoid $\Lambda$ is
an equivalence relation which preserves the addition, that is,
$x \sim x'$ and $y \sim y'$ imply $x + y \sim x' + y'$
for any $x, x', y, y' \in \Lambda$.
Clearly, there exist a smallest additive equivalence relation
which contains a given relation.

An element $\rho$ of an affine monoid $\Lambda$ is \emph{reducible}
if there exist non-zero elements $\sigma$ and $\tau$ of $\Lambda$
satisfying $\sigma + \tau = \rho$.

\subsection{Poincar{\'e} series}

Let $\Lambda$ be an affine monoid,
and fix a field $K$.

\begin{theorem}
  [Laudal-Sletsj{\o}e~\cite{LS}]
  There is an isomorphism
  \begin{equation*}
    \operatorname{Tor}^{K[\Lambda]}_{i, \lambda}(K, K)
    \cong \Tilde H_{i - 2} \bigl( \F(\lambda; \Lambda); K \bigr)
  \end{equation*}
  for each $i \in \N$ and $\lambda \in \Lambda$.
\end{theorem}

The Poincar{\'e} series of $\Lambda$ is defined by
\begin{equation*}
  P_\Lambda(t, z) = \sum_{i \in \N} \sum_{\lambda \in \Lambda}
  \beta_i(\lambda; \Lambda) t^i z^\lambda,
\end{equation*}
where
\begin{equation*}
  \beta_i(\lambda; \Lambda) =
  \dim_K \operatorname{Tor}^{K[\Lambda]}_{i, \lambda}(K, K)
  \quad (i \in \N, \ \lambda \in \Lambda).
\end{equation*}
By the previous theorem, we have
\begin{equation}
  \label{eq:LS}
  \beta_i(\lambda; \Lambda) =
  \Tilde \beta_{i - 2} \bigl( \F(\lambda; \Lambda) \bigr),
\end{equation}
where $\Tilde \beta_i(X)$ denotes the $i$-th reduced Betti number
of a topological space $X$, that is,
\begin{equation*}
  \Tilde \beta_i(X) = \dim_K \Tilde H_i(X; K).
\end{equation*}
We also define
\begin{equation*}
  \Tilde \beta_i(S^{-2}) =
  \begin{cases}
    1 & \text{if $i = -2$,} \\
    0 & \text{otherwise.}
  \end{cases}
\end{equation*}

\begin{lemma}
  Let $X$ and $Y$ be topological spaces.
  Then
  \begin{equation*}
    \Tilde \beta_{i - 2}(X \suspjoin Y) =
    \sum_{j + k = i} \Tilde \beta_{j - 2}(X) \cdot \Tilde \beta_{k - 2}(Y).
  \end{equation*}
\end{lemma}
\begin{proof}
  By \cite[Lemma~2.1]{Mil}, we have
  \begin{align*}
    \Tilde H_{i - 2}(X \suspjoin Y; K)
    &\cong \Tilde H_{i - 3}(X * Y; K) \\
    &\cong \bigoplus_{j + k = i - 4}
    \Tilde H_j(X; K) \otimes \Tilde H_k(Y; K) \\
    &\cong \bigoplus_{j + k = i}
    \Tilde H_{j - 2}(X; K) \otimes \Tilde H_{k - 2}(Y; K).
    \qedhere
  \end{align*}
\end{proof}

\begin{proposition}
  Let $\Lambda_1$ and $\Lambda_2$ be affine monoids.
  Then
  \begin{equation*}
    P_{\Lambda_1 \oplus \Lambda_2}(t, z) =
    P_{\Lambda_1}(t, z) \cdot P_{\Lambda_2}(t, z).
  \end{equation*}
\end{proposition}
\begin{proof}
  By the equation~\eqref{eq:LS}, Lemma~\ref{lem:suspjoin} and
  the previous lemma, we have
  \begin{align*}
    \beta_i(\lambda_1 + \lambda_2; \Lambda_1 \oplus \Lambda_2)
    &= \Tilde \beta_{i - 2} \bigl( \F(\lambda_1 + \lambda_2;
    \Lambda_1 \oplus \Lambda_2) \bigr) \\
    &= \Tilde \beta_{i - 2} \bigl( \F(\lambda_1; \Lambda_1) \suspjoin
    \F(\lambda_2; \Lambda_2) \bigr) \\
    &= \sum_{j + k = i} \Tilde \beta_{j - 2}
    \bigl( \F(\lambda_1; \Lambda_1) \bigr) \cdot
    \Tilde \beta_{k - 2}
    \bigl( \F(\lambda_2; \Lambda_2) \bigr) \\
    &= \sum_{j + k = i} \beta_j(\lambda_1; \Lambda_1)
    \cdot \beta_k(\lambda_2; \Lambda_2).
  \end{align*}
  Thus
  \begin{align*}
    P_{\Lambda_1 \oplus \Lambda_2}(t, z)
    &= \sum_{i \in \N} \sum_{\lambda_1 \in \Lambda_1}
    \sum_{\lambda_2 \in \Lambda_2}
    \beta_i(\lambda_1 + \lambda_2; \Lambda_1 \oplus \Lambda_2)
    t^i z^{\lambda_1 + \lambda_2} \\
    &= \sum_{i \in \N} \sum_{\lambda_1 \in \Lambda_1}
    \sum_{\lambda_2 \in \Lambda_2} \sum_{j + k = i}
    \beta_j(\lambda_1; \Lambda_1) \cdot \beta_k(\lambda_2; \Lambda_2)
    t^i z^{\lambda_1 + \lambda_2} \\
    &= \sum_{j \in \N} \sum_{\lambda_1 \in \Lambda_1}
    \beta_j(\lambda_1; \Lambda_1) t^j z^{\lambda_1}
    \cdot \sum_{k \in \N} \sum_{\lambda_2 \in \Lambda_2}
    \beta_k(\lambda_2; \Lambda_2) t^k z^{\lambda_2} \\
    &= P_{\Lambda_1}(t, z) \cdot P_{\Lambda_2}(t, z).
    \qedhere
  \end{align*}
\end{proof}

\section{The main theorem}

Let $\Lambda_1$ and $\Lambda_2$ be affine monoids,
and let $\rho_1$ and $\rho_2$ be reducible elements of
$\Lambda_1$ and $\Lambda_2$, respectively.
Let $\Lambda$ be the quotient of the direct sum $\Lambda_1 \oplus \Lambda_2$
modulo the smallest additive equivalence relation $\sim$ satisfying
$\rho_1 \sim \rho_2$.
We denote the equivalence class of $\rho_1$ and $\rho_2$ simply by $\rho$,
and define
\begin{align*}
  \Hat \Lambda_1 &=
  \{\, \lambda_1 \in \Lambda_1 \,|\, \lambda_1 \not\ge \rho_1 \,\} \\
  \Hat \Lambda_2 &=
  \{\, \lambda_2 \in \Lambda_2 \,|\, \lambda_2 \not\ge \rho_2 \,\}.
\end{align*}

\begin{proposition}
  \label{prp:properties}
  The following hold.
  \begin{enumerate}
    \item
      The quotient $\Lambda$ has the additive monoid structure
      inherited from $\Lambda_1 \oplus \Lambda_2$.
    \item
      With the above monoid structure
      $\Lambda$ is an affine monoid.
    \item
      For any element $\lambda$ of $\Lambda$
      there uniquely exist $n \in \N$, $\Hat \lambda_1 \in \Hat \Lambda_1$
      and $\Hat \lambda_2 \in \Hat \Lambda_2$ which satisfy
      $n \rho + \Hat \lambda_1 + \Hat \lambda_2 = \lambda$.
  \end{enumerate}
\end{proposition}
\begin{proof}
  The proof is straightforward.
\end{proof}

\begin{theorem}
  The Frobenius complexes of $\Lambda$ satisfy
  \begin{equation*}
    \F(\lambda; \Lambda) \simeq
    \bigvee_{\ell \rho + \lambda_1 + \lambda_2 = \lambda}
    S^{2 \ell - 2} \suspjoin
    \F(\lambda_1; \Lambda_1) \suspjoin \F(\lambda_2; \Lambda_2)
  \end{equation*}
  for $\lambda \in \Lambda$,
  where $\ell$ runs in $\N$, $\lambda_1$ in $\Lambda_1$
  and $\lambda_2$ in $\Lambda_2$.
\end{theorem}

\begin{proof}
  Take $n \in \N$, $\Hat \lambda_1 \in \Hat \Lambda_1$ and
  $\Hat \lambda_2 \in \Hat \Lambda_2$ which satisfy
  $\lambda = n \rho + \Hat \lambda_1 + \Hat \lambda_2 \in \Lambda$.
  Let $P$ be the poset of non-empty subset of $[n] = \{ 0, \dots, n \}$
  ordered by reverse inclusion, that is,
  \begin{equation*}
    p \le q \iff p \supset q
    \quad (p, q \in P).
  \end{equation*}
  Define
  \begin{equation*}
    \Tilde \Lambda =
    \Lambda_1 \oplus \N \alpha_1 \oplus \dots
    \oplus \N \alpha_n \oplus \Lambda_2.
  \end{equation*}
  For $p \in P$ let $\Lambda_p$ be the quotient monoid of $\Tilde \Lambda$
  modulo the additive equivalence relation $\sim_p$ generated by
  \begin{equation*}
    \alpha_i \sim_p \alpha_{i + 1}
    \quad (i \in [n] \setminus p),
  \end{equation*}
  where $\alpha_0$ and $\alpha_{n + 1}$ denote $\rho_1$ and $\rho_2$,
  respectively,
  and let $\Tilde \varphi_p \colon \Tilde \Lambda \to \Lambda_p$ be
  the canonical surjection.
  Note that for $p = \{ p_0 < \dots < p_s \} \in P$ the composition
  \begin{equation}
    \label{eq:isom}
    \Lambda_1 \oplus \N \alpha_{p_1} \oplus \dots
    \oplus \N \alpha_{p_s} \oplus \Lambda_2
    \hookrightarrow \Tilde \Lambda
    \xrightarrow{\Tilde \varphi_p} \Lambda_p
  \end{equation}
  is an isomorphism.
  For $p, q \in P$ with $p \le q$
  let $\varphi^p_q \colon \Lambda_p \to \Lambda_q$ be
  the homomorphism satisfying
  $\varphi^p_q \circ \Tilde \varphi_p = \Tilde \varphi_q$.
  Note that $\varphi^p_q$ is proper.
  Define
  \begin{equation*}
    \Tilde \lambda =
    \Hat \lambda_1 + \alpha_1 + \dots + \alpha_n + \Hat \lambda_2
    \in \Tilde \Lambda,
  \end{equation*}
  and for $p \in P$ let $\lambda_p = \Tilde \varphi_p(\Tilde \lambda)$.
  Then $\varphi^p_q$ sends $\lambda_p$ to $\lambda_q$
  for $p, q \in P$ with $p \le q$.

  Let $X$ and $Y$ be the $P$-diagrams of posets defined as follows.
  For $p \in P$ let
  \begin{align*}
    X(p) &= (0, \lambda_p)_{\Lambda_p} \\
    Y(p) &= C(\lambda_p; \Lambda_p)
  \end{align*}
  and for $p, q \in P$ with $p \le q$
  let $X^p_q \colon X(p) \to X(q)$ and $Y^p_q \colon Y(p) \to Y(q)$ be
  the poset maps induced by $\varphi^p_q$.
  By Proposition~\ref{prp:comp}, two $P$-diagrams $\oc X$ and $\oc Y$
  are naturally homeomorphic.
  By Theorem~\ref{thm:Thomason}, we have
  \begin{equation*}
    \oc{P \ltimes X} \simeq
    \hocolim \oc X \approx
    \hocolim \oc Y \simeq
    \oc{P \ltimes Y}.
  \end{equation*}
  The rest of proof is divided into two parts:
  In the first part, we show that $\oc{P \ltimes Y}$ is homotopy equivalent
  to $\F(\lambda; \Lambda)$.
  In the second part, we show that $\oc{P \ltimes X}$
  is homotopy equivalent to the right-hand side of the theorem.

  Let $\Tilde \pi \colon \Tilde \Lambda \to \Lambda$ be
  the homomorphism defined by
  \begin{align*}
    \Tilde \pi(\lambda_1) &= \lambda_1 && (\lambda_1 \in \Lambda_1) \\
    \Tilde \pi(\alpha_i) &= \rho && (i = 1, \dots, n) \\
    \Tilde \pi(\lambda_2) &= \lambda_2 && (\lambda_2 \in \Lambda_2).
  \end{align*}
  For $p \in P$ let $\pi^p \colon \Lambda_p \to \Lambda$ be
  the homomorphism satisfying
  $\pi^p \circ \Tilde \varphi_p = \Tilde \pi$.
  Then $\Tilde \pi$ is proper and sends $\Tilde \lambda$ to $\lambda$,
  and thus $\pi^p$ is proper and sends $\lambda_p$ to $\lambda$.
  Let $f \colon P \ltimes Y \to C(\lambda; \Lambda)$ be
  the poset map defined by
  \begin{equation*}
    f \bigl( (p, \xi) \bigr) = \pi^p_*(\xi)
    \quad \bigl( (p, \xi) \in P \ltimes Y \bigr).
  \end{equation*}
  We now show that $f$ induces a homotopy equivalence
  by using Theorem~\ref{thm:QFL}.
  Let $\zeta = [ \zeta^{(1)} | \dotsb | \zeta^{(k)} ] \in C(\lambda; \Lambda)$.
  It suffices to show that $\oc{f^{-1}(\ge \zeta)}$ is contractible,
  where $f^{-1}(\ge \zeta)$ is an abbreviation for
  $f^{-1}(C(\lambda; \Lambda)_{\ge \zeta})$.
  Let $(p, \xi) \in f^{-1}(\ge \zeta)$, that is, $\pi^p_*(\xi) \ge \zeta$.
  By Lemma~\ref{lem:going down},
  there uniquely exists $\xi' \in C(\lambda_p; \Lambda_p)$
  which satisfies $\xi' \le \xi$ and $\pi^p_*(\xi') = \zeta$.
  Then $(p, \xi')$ is a maximum element of $f^{-1}(\zeta)^{\le (p, \xi)}$,
  since for $(q, \eta) \in f^{-1}(\zeta)^{\le (p, \zeta)}$
  the uniqueness of $\xi'$ implies $\varphi^q_p(\eta) = \xi'$
  and thus $(q, \eta) \le (p, \xi')$.
  By Lemma~\ref{lem:deform},
  $\oc{f^{-1}(\zeta)}$ is a deformation retract of $\oc{f^{-1}(\ge \zeta)}$.

  Using the bijection of Proposition~\ref{prp:properties}.(3),
  define the map $\rem_1 \colon \Lambda \to \Hat \Lambda_1$
  as the composition
  \begin{equation*}
    \Lambda \cong \N \rho \times \Hat \Lambda_1 \times \Hat \Lambda_2
    \xrightarrow{\mathrm{proj}} \Hat \Lambda_1.
  \end{equation*}
  Using the isomorphism~\eqref{eq:isom} and
  the bijection of Lemma~\ref{lem:remainder},
  also define the map $\rem^p_1 \colon \Lambda_p \to \Lambda_1$
  as the composition
  \begin{equation*}
    \Lambda_p
    \cong \Lambda_1 \oplus \N \alpha_{p_1} \oplus \dots
    \oplus \N \alpha_{p_s} \oplus \Lambda_2
    \xrightarrow{\mathrm{proj}} \Lambda_1
    \cong \N \rho_1 \times \Hat \Lambda_1
    \xrightarrow{\mathrm{proj}} \Hat \Lambda_1
  \end{equation*}
  for $p = \{ p_0 < \dots < p_s \} \in P$.
  Then we can check that
  \begin{equation*}
    \rem_1 \circ \pi^p = \rem^p_1.
  \end{equation*}
  Let
  \begin{equation*}
    \ell = \max \Bigl\{\, \ell \in \N \Bigm|
    \ell \rho_1 \le \sum_{i = 1}^k \rem_1(\zeta^{(i)}) \,\Bigr\},
  \end{equation*}
  and let
  \begin{equation*}
    S = \{\, (p, \xi) \in f^{-1}(\zeta) \mid \ell \in p \,\}.
  \end{equation*}
  We now show that $\oc S$ is a deformation retract of $\oc{f^{-1}(\zeta)}$
  by using Lemma~\ref{lem:deform}.
  Let $(p, \xi) \in f^{-1}(\zeta)$, where $p = \{ p_0 < \dots < p_s \}$
  and $\xi = [ \xi^{(1)} | \dotsb | \xi^{(k)} ]$.
  Then we have
  \begin{equation*}
    \lambda_p
    = \sum_{i = 1}^k \xi^{(i)} \ge \sum_{i = 1}^k \rem^p_1(\xi^{(i)})
    = \sum_{i = 1}^k \rem_1(\zeta^{(i)}) \ge \ell \rho_1.
  \end{equation*}
  Thus at least $\ell$ of $\alpha_1, \dots, \alpha_n$ are
  identified with $\rho_1$,
  which implies $\ell \le p_0$.
  If $\ell = p_0$ holds, then $(p, \xi)$ itself is
  a maximum element of $S^{\le (p, \xi)}$.
  Assume that $\ell < p_0$, and let $\bar p = \{ \ell \} \cup p$.
  Using the isomorphism~\eqref{eq:isom} and Lemma~\ref{lem:remainder},
  we can see that there uniquely exist $\bar \xi^{(i)} \in \Lambda_{\bar p}$
  which satisfies $\varphi^{\bar p}_p(\bar \xi^{(i)}) = \xi^{(i)}$ and
  $\bar \xi^{(i)} \not\ge \rho_1$ for each $i$.
  Then $\bar \xi = [ \bar \xi^{(1)} | \dotsb | \bar \xi^{(k)} ]$ satisfies
  $\bar \xi \in C(\lambda_{\bar p}; \Lambda_{\bar p})$,
  ${\varphi^{\bar p}_p}_*(\bar \xi) = \xi$,
  and thus $(\bar p, \bar \xi) \in S^{\le (p, \xi)}$.
  Moreover, the set
  \begin{equation*}
    \{\,
      \xi' \in C(\lambda_{\bar p}; \Lambda_{\bar p})
      \mid
      {\varphi^{\bar p}_p}_*(\xi') = \xi
    \,\}
  \end{equation*}
  consists of the unique element $\bar \xi$.
  For $(q, \eta) \in S^{\le (p, \xi)}$ we have
  $q \le \bar p$ and ${\varphi^q_{\bar p}}_*(\eta) = \bar \xi$
  which imply $(q, \eta) \le (\bar p, \bar \xi)$.
  Hence $(\bar p, \bar \xi)$ is a maximum element of $S^{\le (p, \xi)}$.
  Thus $\oc S$ is a deformation retract of $\oc{f^{-1}(\zeta)}$.

  Let us construct a maximum element of $S$.
  Using the isomorphism~\eqref{eq:isom} and Lemma~\ref{lem:remainder},
  we can see that there uniquely exists
  $\bar \zeta^{(i)} \in \Lambda_{\{ \ell \}}$
  which satisfies $\pi^{\{ \ell \}}(\bar \zeta^{(i)}) = \zeta^{(i)}$ and
  $\bar \zeta^{(i)} \not\ge \rho_1$ for each $i$.
  Then $\bar \zeta = [ \bar \zeta^{(1)} | \dotsb | \bar \zeta^{(k)} ]$
  satisfies $\bar \zeta \in C(\lambda_{\{ \ell \}}; \Lambda_{\{ \ell \}})$,
  $\pi^{\{ \ell \}}_*(\bar \zeta) = \zeta$,
  and thus $(\{ \ell \}, \bar \zeta) \in S$.
  Moreover, the set
  \begin{equation*}
    \{\, \zeta' \in C(\lambda_{\{ \ell \}}; \Lambda_{\{ \ell \}}) \mid
    \pi^{\{ \ell \}}_*(\zeta') = \zeta \,\}
  \end{equation*}
  consists of the unique element $\bar \zeta$.
  For $(q, \eta) \in S$ we have $q \le \{ \ell \}$ and
  ${\varphi^q_{\{ \ell \}}}_*(\eta) = \bar \zeta$,
  which imply $(q, \eta) \le (\{ \ell \}, \bar \zeta)$.
  Hence $(\{ \ell \}, \bar \zeta)$ is a maximum element of $S$,
  and thus $\oc S$ is contractible.
  From the above we have
  \begin{equation*}
    \oc{f^{-1}(\ge \zeta)} \simeq \oc{f^{-1}(\zeta)} \simeq \oc S \simeq \pt.
  \end{equation*}
  Thus we conclude that
  \begin{equation*}
    \oc{P \ltimes Y} \simeq
    \oc{C(\lambda; \Lambda)} \approx \F(\lambda; \Lambda).
  \end{equation*}

  Next, we consider $P \ltimes X$.
  Define $g \colon P \ltimes X \to P$ by
  \begin{equation*}
    g(p, x) = p
    \quad \bigl( (p, x) \in P \ltimes X \bigr).
  \end{equation*}
  Let us check that $g$ satisfies the assumption of Theorem~\ref{thm:PFT}.
  Clearly, $P$ has the minimum element $[n] = \{ 0, \dots, n \}$.
  In the case $n = 0$ the assumption is trivially satisfied.
  Assume that $n \ge 1$.
  Let $p = \{ p_0 < \dots < p_s \} \in P$ with $p > [n]$.
  Note that
  \begin{gather*}
    \begin{split}
      g^{-1}(P^{< p})
      &= P^{< p} \ltimes X \\
      &= \{\, (q, \mu) \,|\, q \in P^{< p}, \ \mu \in X(q) \,\}
    \end{split} \\
    \begin{split}
      g^{-1}(P^{\le p})
      &= P^{\le p} \ltimes X \\
      &= \{\, (q, \mu) \,|\, q \in P^{\le p}, \ \mu \in X(q) \,\}.
    \end{split}
  \end{gather*}
  Moreover, the poset map $P^{\le p} \ltimes X \to X(p)$
  which sends $(q, \mu)$ to $X^q_p(\mu)$ induces
  a homotopy equivalence,
  since the canonical injection $X(p) \hookrightarrow P^{\le p} \ltimes X$
  induces a homotopy inverse.
  Thus it suffices to show that the composition
  \begin{equation*}
    \oc{P^{< p} \ltimes X} \hookrightarrow
    \oc{P^{\le p} \ltimes X} \to \oc{X(p)}
  \end{equation*}
  is homotopic to a constant map.
  This map is induced by the poset map
  \begin{equation*}
    \varphi \colon P^{< p} \ltimes X \to X(p)
  \end{equation*}
  which sends $(q, \mu)$ to $\varphi^q_p(\mu)$.
  By definition, we have
  \begin{align*}
    \oc{X(p)}
    &= \F(\lambda_p; \Lambda_p)
    \intertext{using the isomorphism~\eqref{eq:isom}}
    &\approx \F(p_0 \rho_1 + \Hat \lambda_1 + (p_1 - p_0) \alpha_{p_1} + \dots
    + (p_s - p_{s - 1}) \alpha_{p_s} + (n - p_s) \rho_2 + \Hat \lambda_2; \\
    &\hspace{19em} \Lambda_1 \oplus \N \alpha_{p_1} \oplus \dots \oplus
    \N \alpha_{p_s} \oplus \Lambda_2)
    \intertext{using Lemma~\ref{lem:suspjoin}}
    &\approx \F(p_0 \rho_1 + \Hat \lambda_1; \Lambda_1)
    \suspjoin \F(p_1 - p_0; \N)
    \suspjoin \dotsb \\
    &\hspace{12em} \dotsb \suspjoin \F(p_s - p_{s - 1}; \N)
    \suspjoin \F((n - p_s) \rho_2 + \Hat \lambda_2; \Lambda_2).
  \end{align*}
  By Proposition~\ref{prp:N},
  if $p = \{ p_0, \dots, p_s \}$ does not have the form
  $\{ \ell_1, \ell_1 + 1, \dots, \ell_1 + \ell \}$,
  then $\oc{X(p)}$ is contractible, and thus
  $\oc \varphi$ is homotopic to a constant map.
  Assume that $p = \{ \ell_1, \ell_1 + 1, \dots, \ell_1 + \ell \}$,
  and let $\ell_2 = n - (\ell_1 + \ell)$.
  Then we have
  \begin{align*}
    \oc{X(p)}
    &\simeq \F(\ell_1 \rho_1 + \Hat \lambda_1; \Lambda_1) \suspjoin
    \underbrace{
      S^{-1} \suspjoin \dotsb \suspjoin S^{-1}
    }_\ell
    \suspjoin \> \F(\ell_2 \rho_2 + \Hat \lambda_2; \Lambda_2) \\
    &\approx S^{\ell - 2} \suspjoin
    \F(\ell_1 \rho_1 + \Hat \lambda_1; \Lambda_1) \suspjoin
    \F(\ell_2 \rho_2 + \Hat \lambda_2; \Lambda_2).
  \end{align*}
  Take non-zero elements $\sigma_1$ and $\tau_1$ of $\Lambda_1$
  satisfying $\sigma_1 + \tau_1 = \rho_1$.
  Similarly, take non-zero elements $\sigma_2$ and $\tau_2$ of $\Lambda_2$
  satisfying $\sigma_2 + \tau_2 = \rho_2$.
  For $q \in P^{< p}$
  let $\psi^q_p \colon \Lambda_q \to \Lambda_p$ be
  the homomorphism defined by
  \begin{align*}
    \psi^q_p(\mu_1) &= \mu_1 \quad (\mu_1 \in \Lambda_1) \\
    \psi^q_p(\alpha_i) &=
    \begin{cases}
      \sigma_1 & \text{if $\min q < i \le \ell_1$} \\
      \sigma_2 & \text{if $\ell_1 + \ell < i \le \max q$} \\
      \alpha_i & \text{otherwise}
    \end{cases} \\
    \psi^q_p(\mu_2) &= \mu_2 \quad (\mu_2 \in \Lambda_2).
  \end{align*}
  Then $\psi^q_p$ is well-defined, proper
  and satisfies $\psi^q_p \le \varphi^q_p$.
  Moreover, $\psi^r_p \le \psi^q_p \circ \varphi^r_q$ holds
  for $q, r \in P$ with $r \le q < p$.
  Define the map $\psi \colon P^{< p} \ltimes X \to X(p)$ by
  \begin{equation*}
    \varphi \bigl( (q, \mu) \bigr) = \psi^q_p(\mu)
    \quad \bigl( (q, \mu) \in P^{< p} \ltimes X \bigr).
  \end{equation*}
  Then $\psi$ is a poset map, and satisfies
  $\psi \le \varphi$ and thus $\oc \psi \simeq \oc \varphi$.
  Therefore it suffices to show that the image of $\oc \psi$ is contractible
  in $\oc{X(p)}$.
  For $q \in P^{< p}$
  either $\min q < \ell_1$ or $\ell_1 + \ell < \max q$ holds.
  By the definition of $\psi^q_p$,
  $\min q < \ell_1$ implies $\psi^q_p(\lambda_q) \le \lambda_p - \tau_1$,
  and $\ell_1 + \ell < \max q$ implies
  $\psi^q_p(\lambda_q) \le \lambda_p - \tau_2$.
  If $\ell_1 = 0$,
  then $\ell_1 + \ell < \max q$ holds for each $q \in P^{< p}$,
  and thus the image of $\oc \psi$ is contained in
  $\oc{(0, \lambda_p - \tau_2]_{\Lambda_p}}$.
  Similarly, in the case $\ell_2 = 0$ the image of $\oc \psi$
  is contained in $\oc{(0, \lambda_p - \tau_1]_{\lambda_p}}$.
  Assume that both $\ell_1$ and $\ell_2$ are positive.
  Then we have
  \begin{align*}
    \image \oc \psi
    &= \oc{\image \psi} \\
    &= \oc{\bigcup_{q < p}
    \psi^p_q \bigl( (0, \lambda_q)_{\Lambda_q} \bigr)} \\
    &\subset \oc{(0, \lambda_p - \tau_1]_{\Lambda_p}}
    \cup \oc{(0, \lambda_p - \tau_2]_{\Lambda_p}}.
  \end{align*}
  Moreover,
  \begin{equation*}
    \oc{(0, \lambda_p - \tau_1]_{\Lambda_p}}
    \cap \oc{(0, \lambda_p - \tau_2]_{\Lambda_p}}
    = \oc{(0, \lambda_p - \tau_1 - \tau_2]_{\Lambda_p}}.
  \end{equation*}
  Since each of $\oc{(0, \lambda_p - \tau_1]_{\Lambda_p}}$,
  $\oc{(0, \lambda_p - \tau_2]_{\Lambda_p}}$ and their intersection
  is contractible, so is their union.
  Hence in any case the image of $\oc \psi$ is
  contained in a contractible set.

  By Theorem~\ref{thm:PFT}, we have
  \begin{align*}
    \oc{P \ltimes X}
    &\simeq \bigvee_{p \in P} \oc{P_{> p}} * \oc{g^{-1}(P^{\le p})} \\
    &\simeq \bigvee_{p \in P} \oc{P_{> p}} * \oc{X(p)}
    \intertext{as seen above}
    &\simeq \bigvee_{\ell_1 + \ell + \ell_2 = n}
    \oc{P_{> \{ \ell_1, \dots, \ell_1 + \ell \}}} *
    S^{\ell - 2} \suspjoin
    \F(\ell_1 \rho_1 + \Hat \lambda_1; \Lambda_1) \suspjoin
    \F(\ell_2 \rho_2 + \Hat \lambda_2; \Lambda_2) \\
    &\simeq \bigvee_{\ell \rho + \lambda_1 + \lambda_2 = \lambda}
    S^{2 \ell - 2} \suspjoin
    \F(\lambda_1; \Lambda_1) \suspjoin
    \F(\lambda_2; \Lambda_2).
  \end{align*}
  Note that $\oc{P_{> \{ \ell_1, \dots, \ell_1 + \ell \}}}$
  is homeomorphic to the boundary of $\ell$-simplex.
\end{proof}

\begin{corollary}
  The Poincar{\'e} series of $\Lambda$ is expressed as
  \begin{equation*}
    P_{\Lambda}(t, z) =
    \frac{P_{\Lambda_1}(t, z) \cdot P_{\Lambda_2}(t, z)}{1 - t^2 z^\rho}.
  \end{equation*}
\end{corollary}
\begin{proof}
  By the previous theorem, we have
  \begin{align*}
    \beta_i(\lambda; \Lambda)
    &= \Tilde \beta_{i - 2} \bigl( \F(\lambda; \Lambda) \bigr) \\
    &= \Tilde \beta_{i - 2} \biggl(
    \bigvee_{\ell \rho + \lambda_1 + \lambda_2 = \lambda}
    S^{2 \ell - 2} \suspjoin \F(\lambda_1; \Lambda_1)
    \suspjoin \F(\lambda_2; \Lambda_2) \biggr) \\
    &= \sum_{\ell \rho + \lambda_1 + \lambda_2 = \lambda}
    \Tilde \beta_{i - 2} \bigl( S^{2 \ell - 2} \suspjoin
    \F(\lambda_1; \Lambda_1) \suspjoin \F(\lambda_2; \Lambda_2) \bigr) \\
    &= \sum_{\ell \rho + \lambda_1 + \lambda_2 = \lambda}
    \sum_{m + j + k = i}
    \Tilde \beta_{m - 2} (S^{2 \ell - 2}) \cdot
    \Tilde \beta_{j - 2} \bigl( \F(\lambda_1; \Lambda_1) \bigr) \cdot
    \Tilde \beta_{k - 2} \bigl( \F(\lambda_2; \Lambda_2) \bigr). \\
    &= \sum_{\ell \rho + \lambda_1 + \lambda_2 = \lambda}
    \sum_{2 \ell + j + k = i}
    \beta_j(\lambda_1; \Lambda_1) \cdot \beta_k(\lambda_2; \Lambda_2).
  \end{align*}
  Thus
  \begin{align*}
    P_\Lambda(t, z)
    &= \sum_{i \in \N} \sum_{\lambda \in \Lambda}
    \beta_i(\lambda; \Lambda) t^i z^\lambda \\
    &= \sum_{i \in \N} \sum_{\lambda \in \Lambda}
    \sum_{\ell \rho + \lambda_1 + \lambda_2 = \lambda}
    \sum_{2 \ell + j + k = i}
    \beta_j(\lambda_1; \Lambda_1) \cdot \beta_k(\lambda_2; \Lambda_2)
    t^i z^\lambda \\
    &= \sum_{\ell \in \N} \sum_{j \in \N} \sum_{k \in \N}
    \sum_{\lambda_1 \in \Lambda_1} \sum_{\lambda_2 \in \Lambda_2}
    \beta_j(\lambda_1; \Lambda_1) \cdot \beta_k(\lambda_2; \Lambda_2)
    t^{2 \ell + j + k} z^{\ell \rho + \lambda_1 + \lambda_2} \\
    &= \sum_{\ell \in \N} t^{2 \ell} z^\rho \cdot
    \sum_{j \in \N} \sum_{\lambda_1 \in \Lambda_1}
    \beta_j(\lambda_1; \Lambda_1) t^j z^{\lambda_1} \cdot
    \sum_{k \in \N} \sum_{\lambda_2 \in \Lambda_2}
    \beta_k(\lambda_2; \Lambda_2) t^k z^{\lambda_2} \\
    &= \frac{P_{\Lambda_1}(t, z) \cdot P_{\Lambda_2}(t, z)}{1 - t^2 z^\rho}.
    \qedhere
  \end{align*}
\end{proof}

\section*{Acknowledgements}
The author wishes to express his thanks to
Prof.\ Satoshi Murai, Prof.\ Takafumi Shibuta and Prof.\ Shunsuke Takagi
for several helpful comments
concerning the Poincar{\'e} series of commutative algebras.

\end{document}